\documentclass[10pt,twocolumn,twoside]{IEEEtran}
\IEEEoverridecommandlockouts

\usepackage{amsmath,epsfig}
\usepackage{amssymb}
\usepackage{amsmath}
\usepackage{cases} 
\usepackage{amssymb,amsmath,cite}
\usepackage{epsfig}
\usepackage{color}
\usepackage{slashbox}
\usepackage[linesnumbered,boxed]{algorithm2e}
\usepackage{bm}

\usepackage{balance}

\renewcommand{\frac}{\dfrac}

\newtheorem{theorem}{Theorem}[section]
\newtheorem{lemma}[theorem]{Lemma}

\newenvironment{proof}[1][Proof]{\begin{trivlist}
\item[\hskip \labelsep {\bfseries #1}]}{\end{trivlist}}

\ifCLASSINFOpdf
\else
\fi
\begin{document}
%
\title{A unified recovery bound estimation for noise-aware $\ell_{q}$ optimization model in compressed sensing
\author{Zhi-Long Dong, Xiao-Qi Yang, Yu-Hong Dai}
\thanks{This work was supported in part by the ******************************, Grant **********, the *********************, Grant ***************, and the *******************************************************.}
 \thanks{Z.-L.~Dong and Y.-H. Dai are with the State Key Laboratory
of Scientific and Engineering Computing, Institute of Computational
Mathematics and Scientific/Engineering Computing, Academy of
Mathematics and Systems Science, Chinese Academy of Sciences,
Beijing, 100190, China (e-mail:
{\{{zldong,\,dyh}\}@lsec.cc.ac.cn}).}
\thanks{X.-Q.~Yang is with the Department of Applied Mathematics, The Hong Kong Polytechnic University, Kowloon, Hong Kong (e-mail:
 {mayangxq}@polyu.edu.hk).}}
\maketitle

\begin{abstract}
 In this letter, we present a unified result for the stable recovery bound of $\ell_q$ $(0<q\leq1)$ optimization model in compressed sensing, which is a constrained $\ell_q$ minimization problem aware of the noise in a linear system. Specifically, without using the restricted isometry constant (RIC), we show that the error between any global solution of the noise-aware $\ell_q$ optimization model and the ideal sparse solution of the noiseless model is upper bounded by a constant times the noise level, given that the sparsity of the ideal solution is smaller than a certain number. An interesting parameter $\gamma$ is introduced, which indicates the sparsity level of the error vector and plays an important role in our analysis. In addition, we show that when $\gamma > 2$, the recovery bound of the $\ell_q$ $(0<q<1)$ model is smaller than that of the $\ell_1$ model, and the sparsity requirement of the ideal solution in the $\ell_q$ $(0<q<1)$ model is weaker than that of the $\ell_1$ model.
\end{abstract}



\section{Introduction}
The problem of finding the sparsest representation possibility in an overcomplete dictionary $A$ is
\begin{equation}\label{P0}
(P_0):\begin{aligned}
\min_{x}&~\|x\|_0\\
 \text{s.t.}&~y_0=Ax,
\end{aligned}
\end{equation}
where $y_0\in\mathbb R^n, {A}\in\mathbb R^{m\times n}(m\ll n), x\in\mathbb R^n$, $\|x\|_0$ means the number of nonzero entries in $x$. In this letter, we assume that problem \eqref{P0} always has a unique ideal sparsest solution $x_0$, and the sparsity level defined by $N=\|x_0\|_0$ is small. This problem has applications in many areas, such as statistics and signal processing, etc \cite{eldar2012compressed}. However, it is well known that problem \eqref{P0} is NP-hard due to its combinatorial nature \cite{L0NP95}. Moreover, it is easy to notice that each feasible point of \eqref{P0} is a local minimizer, which makes it difficult to find the global solution.

A popular way to solve the original problem \eqref{P0} is to approximate it by replacing the $\ell_0$ norm with the convex $\ell_1$ norm
\begin{equation}\label{P1}
(P_1):\begin{aligned}
\min_{x}&~\|x\|_1\\
 \text{s.t.}&~y_0=Ax.
\end{aligned}
\end{equation}
A variety of results have shown that if the RIC of matrix $A$ satisfies certain condition, problem \eqref{P1} can recover the optimal signal in \eqref{P0} exactly \cite{TONY13,Candes05,DONOHO06E}.

In practice, it is more convincible to formulate the original problem \eqref{P0} into a noise-aware version. Actually, we can only observe a noisy version $y=y_0+w$, where $\omega\in\mathbb R^n$ is the noise, and $||\omega||_2\leq\epsilon$. Then we obtain a new problem
\begin{equation}\label{P0D}
(P_{0,\sigma}):\begin{aligned}
\min_{x}&~\|x\|_0\\
 \text{s.t.}&~||y-{A}x||_2\leq\sigma,
\end{aligned}
\end{equation}
where $\sigma$ is the estimation of the noise level $\epsilon$. A similar convexification strategy for \eqref{P0D} leads to
\begin{equation}\label{P1D}
(P_{1,\sigma}):\begin{aligned}
\min_{x}&~\|x\|_1\\
 \text{s.t.}&~||y-{A}x||_2\leq\sigma.
\end{aligned}
\end{equation}
Researchers in \cite{Candes05,Candes06,TONY13} have shown that there is a stable recovery bound for the solution to \eqref{P1D}, which is defined by the RIC and the noise level $\epsilon$, as well as the sparsity level $N$.

All the above results depend on RIC. However, as is well known, the computation of RIC for a given matrix is difficult \cite{RIP13}. In view of this, David L. Donoho and his collaborators proposed another way to show the stable recovery results for the sparse signal recovery problem constrained by a noise-aware overcomplete system \cite{DONOHO06}. They established the stable recovery bound defined by the {\it mutual coherence} constant of the dictionary, the sparsity level $N$, as well as the noise level $\epsilon$. 
Assume that the dictionary $A=[A_1,A_2,...,A_n]$ has normalized columns under $l^2$-norm, which means $\|A_i\|_2=1,\forall i=1,2,...,n$, $A_i$ is the $i$-th column of matrix $A$. The {\it mutual coherence} constant is defined as
\begin{equation}\label{mutualco}
M=M(A)=\max_{1\leq i,j\leq n,i\neq j}A_i^TA_j,
\end{equation}
thus $0\leq M\leq1$. Smaller $M$ means the dictionary is more incoherent. Compared with RIC, we can see that $M$ is much easier to compute for a given matrix $A$.

Based on the above definition, the authors in \cite{DONOHO06} established the following results:\\
1. Suppose $x_{0,\sigma}^*$ is the optimal solution to problem \eqref{P0D}, if $N<(M^{-1}+1)/2$, then
\begin{equation}
\|x_{0,\sigma}^*-x_0\|_2\leq C_0(M,N)*(\epsilon+\sigma),~~\forall \sigma\geq\epsilon>0,
\end{equation}
where the coefficient $C_0(M,N)=\sqrt{1/(1-M(2N-1))}$.\\
2. Suppose $x_{1,\sigma}^*$ is the optimal solution to problem \eqref{P1D}, if $N<(M^{-1}+1)/4$, then
\begin{equation}\label{P1Re}
\|x_{1,\sigma}^*-x_0\|_2\leq C_1(M,N)*(\epsilon+\sigma),~~\forall \sigma\geq\epsilon>0,
\end{equation}
where the coefficient $C_1(M,N)=\sqrt{1/(1-M(4N-1))}$.

At the same time, researchers have also found that, in some cases, the non-convex model could be more efficient in sparse signal recovery problems compared to \eqref{P1D} \cite{Chartrand08,Mingjunlai09,Saab10,Mourad10,yafeng2016}. There are a lot of works considering the sharp RIC condition for exact recovery of noiseless $\ell_q$ $(0<q<1)$ minimization problem, as well as the stable recovery bound for noise-aware $\ell_q$ $(0<q<1)$ minimization problem \cite{Saab08,Jinming14lq,song2014sparse}. These stable recovery bound has the same difficulty in computing the RIC of a given dictionary. Thus, following Donoho's work, we want to give a similar stable recovery bound for $\ell_q$ $(0<q<1)$ minimization problem without relying on RIC.

The main contribution of this letter is that we propose a unified stable recovery bound for the non-convex noise-aware model in compressed sensing. Specifically, consider using a non-convex relaxation strategy for $P_{0,\sigma}$, which replaces the $\ell_0$ norm with the $\ell_q$ $(0<q\leq1)$ quasi-norm. Minimize the $\ell_q$ quasi-norm of the coefficient vector, subject to a noise-aware linear system
\begin{equation}\label{PqD}
(P_{q,\sigma}):\begin{aligned}
\min_{x}&~\|x\|_{q}\\
 \text{s.t.}&~~||y-{A}x||_2\leq\sigma,
\end{aligned}
\end{equation}
where $\|x\|_q=(\sum |x_i|^q)^{1/q},~0<q\leq1$. Denote the solution set of problem \eqref{PqD} as $\mathcal X$, we present a unified estimation about  the upper bound of the error $||x^*_{q,\sigma}-x_0||_2, \forall x^*_{q,\sigma}\in\mathcal X$. If $N<\frac{\gamma^{2/q-2}(M+1)}{4^{1/q}M}$, then
\begin{equation}
\|x^*_{q,\sigma}-x_0\|_2\leq
C_q(M,N,\gamma)(\epsilon+\sigma), \forall x^*_{q,\sigma}\in\mathcal X,
\end{equation}
where $C_q(M,N,\gamma)=\sqrt{1/(1-
M(\frac{4^{1/q}}{\gamma^{2/q-2}}N-1))}$. $M$ and $N$ are the same as former definitions, $\gamma$ is a constant lies in the interval $[1/N,n/N]$, which is defined later. Further analysis shows that, when $\gamma>2$, the $\ell_q$ $(0<q<1)$ model outperforms the $\ell_1$ model in the sense that the former one has a smaller stable recovery bound, and the requirement of the sparsity $N$ is weaker than the latter one.

The rest of this letter is organized as follows. Section II presents several basic lemmas that are needed in subsequent analysis. In Section III, we show the main theorem and some analysis about the result. Section IV concludes the letter and makes some comments on future work.
\section{Some basic lemmas}

First, we recall a well-known result without showing the proof \cite{YaoHua16}.
\begin{lemma}\label{pq}
$\|x\|_p\leq\|x\|_q, \forall x\in\mathbb R^n,0<p\leq q$.
\end{lemma}
With this result, we can prove an important  lemma in this letter.
\begin{lemma}\label{ineqq}
$|a+b|^q+|b|^q\geq|a|^q, \forall a,b\in\mathbb R,0<q\leq1$.
\end{lemma}
\begin{proof}
It is obvious that
\begin{equation}
(|a+b|+|b|)^q\geq|a|^q,
\end{equation}
thus, we only need to prove
\begin{equation}
|a+b|^q+|b|^q\geq(|a+b|+|b|)^q.
\end{equation}
Due to lemma \eqref{pq}, we denote a vector $x=(|a+b|,|b|)$ and $p=1$, then the lemma holds.
\end{proof}

Here is a lemma showing the relationship between the $\ell_1$ norm and the $\ell_q$ quasi-norm of error vector $e$.
\begin{lemma}\label{l1q}
Let $e=x^*_{q,\sigma}-x_0$,
 then
 \begin{equation}
(\gamma N)^{1/q-1}\|e\|_1\leq
\|e\|_{q}\leq n^{1/q-1}\|e\|_1,
\forall 0<q\leq1,
 \end{equation}
where $N=\|x_0\|_0$, $\gamma$ lies in the interval $[1/N,n/N]$, $n$ is the dimension of $e$.
\end{lemma}
\begin{proof}
We first prove the second inequality. Based on the definition, we know that
\begin{equation}
\begin{aligned}
\|e\|_q^q&=\sum_{i=1}^n|e_i|^q\\
&= \sum_{i=1}^n|e_i|^q\cdot 1\\
&\leq(\sum_{i=1}^n(|e_i|^q)^{\frac{1}{q}})^{q}(\sum_{i=1}^n
1^{\frac{1}{1-q}})^{1-q}\\
&=\|e\|_1^qn^{1-q}
\end{aligned}
\end{equation}
where the inequality holds due to the Holder's inequality, thus $\|e\|_q\leq n^{1/q-1}\|e\|_1$.

From lemma \eqref{pq}, we know that $\|e\|_1\leq\|e\|_q,\forall 0<q\leq1$, thus $\gamma\geq1/N$. When $e$ has only one nonzero entry, then the first equality in the lemma holds for any $q\in(0,1]$ if and only if $\gamma=1/N$. From the second inequality, we know that $\gamma\leq n/N$. When $e$ has all equal nonzero entries, then the first equality in the lemma holds for any $q\in(0,1]$ if and only if $\gamma=n/N$. Thus $\gamma$ lies in the interval $[1/N,n/N]$.
\end{proof}

Note that $\gamma$ here is a measure of the sparsity level of the error vector $e$. In general, $\gamma$ increases as the number of nonzero entries in $e$ increases. 

We end this section with the following well-known fact.
\begin{lemma}\label{l12}
For any vector $x\in\mathbb R^n$, we have
\begin{equation}
\|x\|_2\leq\|x\|_1\leq\sqrt{n}\|x\|_2.
\end{equation}
\end{lemma}

\section{Main results and the analysis}
In this section, we present the main result of the stable recovery bound for the noise-aware non-convex minimization problem \eqref{PqD}. This result extends  Donoho and his coauthors' result in \cite{DONOHO06} to the general case where $q\in(0,1]$. 
We first present the main theorem and the details of the proof and then make our comments.
\subsection{Main results}
Before the main theorem, we show a lemma about the optimal solutions to problem \eqref{PqD}.
\begin{lemma}\label{optlemma}
For the optimal solution to problem \eqref{PqD}, the inequality in the constraint is active.
\end{lemma}
\begin{proof}
In order to prove this result, we rewrite the problem \eqref{PqD} into an equivalent form
\begin{equation}\label{PqD2}
\begin{aligned}
\min_{x}&~\|x\|_{q}^q\\
 \text{s.t.}&~||y-{A}x||_2^2\leq\sigma^2,\\
 &~0<q\leq1.
\end{aligned}
\end{equation}
The optimal solution to problem \eqref{PqD2} also solves problem \eqref{PqD}. Assume that $x^*_{q,\sigma}$ is a local minimizer of problem \eqref{PqD2}, and the support set of $x^*_{q,\sigma}$ is $\mathcal S$, $x^*_{q,\sigma}$ is also the stationary point. Then the first-order optimality conditions for this problem holds if there exists a pair of $(x^*_{q,\sigma},\lambda^*)$ such that
\begin{equation}\label{KKT}
\begin{aligned}
\nabla(\|(x_{q,\sigma}^*)_{\mathcal S }\|_{q}^{q})-\lambda^*{A_{\mathcal S}^T}(y_{\mathcal S}-{A_{\mathcal S}}(x_{q,\sigma}^*)_{\mathcal S})&=0,\\
\lambda^*(\|y-{A}x_{q,\sigma}^*\|_2^2-\sigma^2)&=0,\\
\|y-{A}x_{q,\sigma}^*\|_2^2&\leq\sigma^2,\\
\lambda^*&\geq0.
\end{aligned}
\end{equation}
From the first equality in \eqref{KKT}, we know that if $\lambda^*=0$, then $x_{q,\sigma}^*$ should be a zero vector. The zero vector may not satisfy the third inequality in the KKT system \eqref{KKT}, since $y$ is a nonzero observation vector with noise. Thus, we do accept the fact that $\lambda^*>0$. And then according to the complementarity condition, we will have $\|y-{A}x_{q,\sigma}^*\|_2^2=\sigma^2$, i.e., the conclusion of the lemma holds.
\end{proof}

We are now ready to present our main result.
\begin{theorem}\label{lqnormbound}
Suppose the ideal sparse representation signal satisfies
\begin{equation}\label{lqresult}
 N:=\|x_0\|_0<\frac{\gamma^{2/q-2}(M+1)}{4^{1/q}M},
\end{equation}
where M is defined in \eqref{mutualco} and $\gamma$ is defined in Lemma \eqref{l1q}. Then the difference between any solution to $P_{q,\sigma}$ and $x_0$, assuming $\sigma\geq\epsilon$, is upper bounded
\begin{equation}\label{reBound}
\|x^*_{q,\sigma}-x_0\|_2^2\leq
\frac{(\epsilon+\sigma)^2}
{{1-M(\frac{4^{1/q}}{\gamma^{2/q-2}}N-1)}}, \forall x^*_{q,\sigma}\in\mathcal X.
\end{equation}
\end{theorem}

\begin{proof}
Let $x^*$ be the optimal solution to problem \eqref{PqD}, denote the error as $e = x^*-x_0$, the process to estimate the bound can be formulated into an optimization problem in this form
\begin{equation}\label{q}
\begin{aligned}
\max~&~\|e\|_2^2\\
\text{s.t.}~ &~\mathcal X:=\arg\min_{x}\{\|x\|_{q}^{q}:\|{A}x-y\|_2\leq\sigma\},\\
&~x^*\in\mathcal X,\\
&~y={A}x_0+w,\\
&~\|w\|_2\leq\epsilon,\|x_0\|_0=N.
\end{aligned}
\end{equation}
The constraints indicate that we consider the worst case since all the optimal solutions to problem \eqref{PqD} are taken into consideration. Define the optimal function value of \eqref{q} as Val\eqref{q}. The upper bound of Val\eqref{q} is the target, and the main idea of the proof is to expand the feasible set of problem \eqref{q} sequentially and obtain a series of upper bound on Val\eqref{q}.

Note that $x_0$ itself is a feasible point of the constraint $\|{A}x-y\|_2\leq\sigma$. Based on the conclusion of Lemma \eqref{optlemma}, do some simple substitution and eliminate $y$, we can relax the constraints in \eqref{q} to
\begin{equation}\label{omegaq2}
\left\{e\biggl|
\begin{aligned}
&\|x_0+e\|_{q}^{q}\leq\|x_0\|_{q}^{q},
\|{A}e-w\|_2=\sigma\\
&\|w\|_2\leq\epsilon,\|x_0\|_0=N
\end{aligned}
\right\}.
\end{equation}
Define $\mathcal S=\{i|(x_0)_i\neq0\}$ as the support set and the complement set as $\mathcal S^c=\{1,2,...,n\}\backslash\mathcal S$. From Lemma \eqref{ineqq} we know that
\begin{equation}
\begin{aligned}
\|x_0+e\|_{q}^{q}-\|x_0\|_{q}^{q}
\geq&\|e_{\mathcal S^c}\|_{q}^{q}-\|e_{\mathcal S}\|_{q}^{q}
=\|e\|_{q}^{q}-2\sum_{i\in \mathcal S}|e_i|^{q}.
\end{aligned}
\end{equation}
Therefore, we can relax the feasible set further and rewrite the problem \eqref{q} as
\begin{equation}\label{omegaq3}
\begin{aligned}
\max_{e,\mathcal S,w}&~\|e\|_2^2\\
\text{s.t.}&~\|e\|_{q}^{q}\leq2\sum_{i\in \mathcal S}|e_i|^{q},\\
&~\|{A}e-w\|_2=\sigma,\\
&~\|w\|_2\leq\epsilon,\#\mathcal S=N,
\end{aligned}
\end{equation}
which has a larger optimal function value than \eqref{q}.

Using the fact that $\|Ae-w\|_2\geq\|Ae\|_2-\|w\|_2$,
we can relax the feasible set in \eqref{omegaq3} and obtain
\begin{equation}\label{omegaq4}
\begin{aligned}
\max_{e,\mathcal S}&~\|e\|_2^2\\
\text{s.t.}&~\|e\|_{q}^{q}\leq2\sum_{i\in \mathcal S}|e_i|^{q},\\
&~\|{A}e\|_2\leq\Sigma,\#\mathcal S=N,
\end{aligned}
\end{equation}
where $\Sigma=\epsilon+\sigma$.

Now it's turn to deal with the constraint $\|{A}e\|_2\leq\Sigma$. Recall the definition of the {\it mutual coherence}, $M = M({A})=max_{i\neq j}|G(i,j)|$, where $G$ is the Gram matrix $G=A^TA$. For any matrix $X$, denote $|X|$ as the matrix after taking absolute values for all the entries in $X$, this include the vector and the constant as the special case. Denote {\bf 1} as a square matrix with all entries equal to one, and the dimension is correspondingly adaptive. Then we can relax the constraint as follows
\begin{equation}
\begin{aligned}
\Sigma^2\geq&\|{A}e\|_2^2=e^T{\bf G}e=\|e\|_2^2+e^T{\bf (G-I)}e\\
\geq&\|e\|_2^2-|e|^T{\bf |G-I|}|e|\\
\geq&\|e\|_2^2-M|e|^T{\bf |1-I|}|e|\\
=&(1+M)\|e\|_2^2-M\|e\|_1^2\\
\geq&(1+M)\|e\|_2^2-\frac{M}{(\gamma N)^{2/q-2}}\|e\|_{q}^2,\\
\end{aligned}
\end{equation}
where the last inequality is due to Lemma \eqref{l1q}. Thus, we can obtain the following problem with a larger optimal function value
\begin{equation}\label{omegaq5}
\begin{aligned}
\max_{e,\mathcal S}&~\|e\|_2^2\\
\text{s.t.}&~\|e\|_{q}^{q}\leq2\sum_{i\in \mathcal S}|e_i|^{q},\\
&~(1+M)\|e\|_2^2-\frac{M}{(\gamma N)^{2/q-2}}\|e\|_{q}^{2}\leq\Sigma^2,\#\mathcal S=N.
\end{aligned}
\end{equation}

Now we can see that the target term of the objective function appears in the second constraint. Our task is to replace the other terms in the constraints with the target term and obtain an upper bound for the function value. First, we separate the error vector into two parts, $e = (e_{\mathcal S},e_{\mathcal S^c})$, where $e_{\mathcal S}=\{e_i|i\in\mathcal S\}$ and $e_{\mathcal S^c}=\{e_i|i\in\mathcal S^c\}$. Then it's obvious that
\begin{equation}
\begin{aligned}
\|e\|_2^2=&\|e_{\mathcal S}\|_2^2+\|e_{\mathcal S^c}\|_2^2,\\
\|e\|_{q}^{2}=&(\|e_{\mathcal S}\|_{q}^{q}+\|e_{\mathcal S^c}\|_{q}^{q})^{2/q}.
\end{aligned}\end{equation}
Based on Lemma\eqref{l1q} and \eqref{l12}, we have
\begin{equation}
\begin{aligned}
\|e_{\mathcal S}\|_2&\leq\|e_{\mathcal S}\|_{q}\leq N^{(1/q-1/2)}\|e_{\mathcal S}\|_2,\\
\|e_{\mathcal S^c}\|_2&\leq\|e_{\mathcal S^c}\|_{q}\leq (n-N)^{(1/q-1/2)}\|e_{\mathcal S^c}\|_2.
\end{aligned}
\end{equation}
In order to make the symbol simpler, we define the following intermediate variables
\begin{equation}
\begin{aligned}
V_{\mathcal S}=&\|e_{\mathcal S}\|_{q}^{q},V_{\mathcal S^c}=\|e_{\mathcal S^c}\|_{q}^{q},\\
\eta_{\mathcal S}=&\|e_{\mathcal S}\|_2^2/\|e_{\mathcal S}\|_{q}^2 ,\eta_{\mathcal S^c}=\|e_{\mathcal S^c}\|_2^2/\|e_{\mathcal S^c}\|_{q}^2,
\end{aligned}
\end{equation}
and reformulate the optimization problem \eqref{omegaq5} into an optimization problem on $(V_{\mathcal S},V_{\mathcal S^c},\eta_{\mathcal S},\eta_{\mathcal S^c})\in R^4$
\begin{equation}\label{ABq1}
\begin{aligned}
\max &~\eta_{\mathcal S}V_{\mathcal S}^{2/q}+\eta_{\mathcal S^c}V_{\mathcal S^c}^{2/q}\\
\text{s.t.}&~(1+M)(\eta_{\mathcal S}V_{\mathcal S}^{2/q}+\eta_{\mathcal S^c}V_{\mathcal S^c}^{2/q})\\
&-\frac{M}{(\gamma N)^{2/q-2}}(V_{\mathcal S}+V_{\mathcal S^c})^{2/q}\leq\Sigma^2,\\
&~0\leq V_{\mathcal S^c}\leq V_{\mathcal S},N^{1-2/q}\leq \eta_{\mathcal S}\leq1,0<\eta_{\mathcal S^c}\leq1.
\end{aligned}
\end{equation}
Define the ratio between $V_{\mathcal S}$ and $V_{\mathcal S^c}$ as $\tau=V_{\mathcal S^c}/V_{\mathcal S}$, then we have $0\leq\tau\leq1$. \eqref{ABq1} can be reformulated as
\begin{equation}\label{ABq2}
\begin{aligned}
\max &~(\eta_{\mathcal S}+\eta_{\mathcal S^c}\tau^{2/q})V_{\mathcal S}^{2/q}\\
\text{s.t.}&~(1+M)(\eta_{\mathcal S}+\eta_{\mathcal S^c}\tau^{2/q})V_{\mathcal S}^{2/q}\\
&-\frac{M}{(\gamma N)^{2/q-2}}(1+\tau)^{2/q}V_{\mathcal S}^{2/q}\leq\Sigma^2,\\
&~V_{\mathcal S}\geq0,N^{1-2/q}\leq \eta_{\mathcal S}\leq1,0<\eta_{\mathcal S^c}\leq1,0\leq\tau\leq1.
\end{aligned}
\end{equation}

Let $\mu=(1+\tau)^{2/q}/(\eta_{\mathcal S}+\eta_{\mathcal S^c}\tau^{2/q})$, then it is easy to see that $1\leq\mu\leq4^{1/q}N^{2/q-1}$ under the constraints in \eqref{ABq2}. Denote the objective function value as $V=V_{\mathcal S}^{2/q}(\eta_{\mathcal S}+\eta_{\mathcal S^c}\tau^{2/q})$, then the first constraint in \eqref{ABq2} is equivalent to
\begin{equation}
 (1+M)V-\frac{M}{(\gamma N)^{2/q-2}}\mu V\leq\Sigma^2.
\end{equation}
Substitute the upper bound of $\mu$ into this inequality, and to maintain the coefficient of $V$ to be positive, we will have the requirement of sparsity $N$ as follows
\begin{equation}
(1+M)-\frac{M}{(\gamma N)^{2/q-2}}\mu
\geq(1+M)-M\frac{4^{1/q}N}{\gamma^{2/q-2}}>0.
\end{equation}
Thus, we have
\begin{equation}
 N<\frac{\gamma^{2/q-2}(M+1)}{4^{1/q}M},
\end{equation}
and
\begin{equation}
 V\leq\frac{\Sigma^2}{1-
 M(\frac{\mu}{(\gamma N)^{2/q-2}}-1)} \leq\frac{\Sigma^2}{1-
 M(\frac{4^{1/q}N}{\gamma^{2/q-2}}-1)},
\end{equation}
which concludes the proof.
\end{proof}

\subsection{On the parameter $\gamma$ in the recovery bound \eqref{reBound}}

From the above results, we can see that, the parameter $\gamma$ plays an important role in the recovery bound \eqref{reBound}. We will remark that for which $\gamma$ the recovery bound of model \eqref{PqD} is better than that of model \eqref{P1D}.

%

{\bf Case 1:} if $\gamma=2$,
we have $C_q(M,N,\gamma)=C_1(M,N)$. Thus the sparsity requirement in \eqref{PqD} is equal to that in \eqref{P1D}, and the recovery bound of model \eqref{PqD} is the same to that of model \eqref{P1D}. This leads to the conclusion that the non-convex minimization model \eqref{PqD} approximates the original model \eqref{P0} no worse than the convex minimization model \eqref{P1D}.

{\bf Case 2:} if $\gamma>2$,
we have $C_q(M,N,\gamma)<C_1(M,N)$. Thus the sparsity requirement in \eqref{PqD} is weaker than that in \eqref{P1D}, and the recovery bound of model \eqref{PqD} is smaller than that of model \eqref{P1D}. This means that the non-convex minimization model \eqref{PqD} is better than the convex minimization model \eqref{P1D} in approximating the original model \eqref{P0}.

For example, let $q=1/2$. From Lemma \eqref{l1q}, we have
\begin{equation}\label{case11}
\gamma N\|e\|_1\leq\|e\|_{1/2}.
\end{equation}
Since $\gamma>2$, from simple calculation, we can see that \eqref{case11} holds if
\begin{equation}\label{case32}
\max_{i=1,2,...,n}|e_i|^{1/2}\leq\frac{\sum_{j\neq i}|e_j|^{1/2}}{\gamma N-1},
\end{equation}
i.e.
\begin{equation}\label{case33}
2<\gamma\leq\frac{\sum_{j}|e_j|^{1/2}}{N*\max_{i=1,2,...,n}|e_i|^{1/2}}.
\end{equation}
 Condition \eqref{case33} requires that the number of nonzero entries in error vector $e$ should be larger than $\gamma N$ in this case. This requirement seems to be very strong, since it says that the worst solution in the optimal solution set $\mathcal X$ has at least $N$ tails except the right sparse locations. However, recent numerical experiments indicate that this could be true in the noise-aware case; see for example \cite{TK16}. Thus, it is conceivable that for the noise-aware models \eqref{P1D} and \eqref{PqD}, \eqref{PqD} outperforms \eqref{P1D} in approximating \eqref{P0} in the sense that \eqref{PqD} has a smaller worst-case stable recovery bound and weaker requirement for the sparsity of the ideal signal.

\section{Conclusion}
In this letter, we present a unified worst-case stable recovery bound for the noise-aware $\ell_q$ $(0<q\leq1)$ minimization problem. The result reduces to the original result in Donoho's paper \cite{DONOHO06} when $q=1$. We introduce a parameter $\gamma$ which describes the relationship between the $\ell_1$ norm and the $\ell_q$ quasi-norm of the error vector $e$. Further analysis indicates that, for $\gamma>2$, the non-convex $\ell_q$ $(0<q<1)$ minimization model outperforms the convex $\ell_1$ model, in the sense that, the former model has a smaller worst-case recovery bound and weaker sparsity requirement.

Moreover, the recovery bound for \eqref{PqD} proposed in this letter is very loose due to the unduely relaxation, similar to the recovery bound for \eqref{P1D} in \cite{DONOHO06}. 
 Even so, the result in this letter still has some guiding significance. Since it is easier to compute the constant in the recovery bound \eqref{reBound} for a given problem. We are also looking forward to the extension of this result to the general sparse group sparse optimization problem\cite{simon2013sparse,vincent2014sparse}, which is also popular in the recent period.


\balance
\bibliographystyle{unsrt}
\bibliography{Lq_IEEE}

\end{document}